\documentclass{article}%
\usepackage{graphicx}
\usepackage[intlimits]{amsmath}
\usepackage{latexsym}
\usepackage{amsfonts}
\usepackage{amssymb}%
\makeatletter
\newfont{\footsc}{cmcsc10 at 8truept}
\newfont{\footbf}{cmbx10 at 8truept}
\newfont{\footrm}{cmr10 at 10truept}
\newtheorem{theorem}{Theorem}

\newtheorem{corollary}[theorem]{Corollary}

\newtheorem{definition}[theorem]{Definition}

\newtheorem{lemma}[theorem]{Lemma}

\newtheorem{remark}[theorem]{Remark}

\newenvironment{proof}[1][Proof]{\noindent{\textbf {#1}  }}  {\hfill$\Box$\bigskip}

\begin{document}

\title{Some bounds on the eigenvalues of uniform hypergraphs}
\author{Xiying Yuan\thanks{Department of Mathematics, Shanghai University, Shanghai
200444, China; \textit{email: xiyingyuan2007@hotmail.com }}, Man
Zhang\thanks{Department of Mathematics, Shanghai University, Shanghai 200444,
China.}, Mei Lu\thanks{Corresponding author. Department of Mathematical
Sciences, Tsinghua University, Beijing, 100084; \textit{email:}
mlu@math.tsinghua.edu.cn}}
\maketitle

\begin{abstract}
Let $\mathcal{H}$ be a uniform hypergraph. Let $\mathcal{A(H)}$ and
$\mathcal{Q(H)}$ be the adjacency tensor and the signless Laplacian tensor of
$\mathcal{H}$, respectively. In this note we prove several bounds for the
spectral radius of $\mathcal{A(H)}$ and $\mathcal{Q(H)}$ in terms of the
degrees of vertices of $\mathcal{H}.$



\end{abstract}

\section{Introduction}

We denote the set $\{1,2,\cdot\cdot\cdot,n\}$ by $[n].$ Hypergraph is a
natural generalization of simple graph (see \cite{Berge}). A hypergraph
$\mathcal{H}=(V(\mathcal{H}),E(\mathcal{H}))$ on $n$ vertices is a set of
vertices, say $V(\mathcal{H})=\{1,2,\cdot\cdot\cdot,n\}$ and a set of edges,
say $E(\mathcal{H})=\{e_{1},e_{2},\cdot\cdot\cdot,e_{m}\},$ where
$e_{i}=\{i_{1},i_{2},\cdots,i_{l}\},i_{j}\in\lbrack n],$ $j=1,2,\cdots,l.$ If
$|e_{i}|=k$ for any $i=1,2,\cdot\cdot\cdot,m,$ then $\mathcal{H}$ is called a
$k$-uniform hypergraph. The degree $d_{i}$ of vertex $i$ is defined as
$d_{i}=|\{e_{j}:i\in e_{j}\in E(\mathcal{H})\}|.$ If $d_{i}=d$ for any vertex
$i$ of hypergraph $\mathcal{H},$ then $\mathcal{H}$ is called a $d$-regular
hypergraph. An order $k$ dimension $n$ tensor $\mathcal{T=}(\mathcal{T}%
_{i_{1}i_{2}\cdots i_{k}})\in\mathbb{C}^{n\times n\times\cdots\times n}$ is a
multidimensional array with $n^{k}$ entries, where $i_{j}\in\lbrack n]$ \ for
each $j=1,2,\cdot\cdot\cdot,k.$ To study the properties of uniform hypergraphs
by algebraic methods, adjacency matrix and signless Laplacian matrix of graph
are generalized to adjacency tenor and signless Laplacian tensor of uniform hypergraph.

\begin{definition}
\cite{HuQiXie} \cite{Qi2014}. Let $\mathcal{H}=(V(\mathcal{H}),E(\mathcal{H}%
))$ be a $k$-uniform hypergraph on $n$ vertices. The adjacency tensor of
$\mathcal{H}$ is defined as the $k$-$th$ order $n$-dimensional tensor
$\mathcal{A}(\mathcal{H})$ whose $(i_{1}\cdots i_{k})$-entry is:
\[
(\mathcal{A}(\mathcal{H}))_{i_{1}i_{2}\cdots i_{k}}=%
\begin{cases}
\frac{1}{(k-1)!} & \text{$\{i_{1},i_{2},\cdots,i_{k}\}\in E(\mathcal{H})$}\\
0 & \text{otherwise}.
\end{cases}
\]
Let $\mathcal{D(\mathcal{H}})$ be a $k$-$th$ order $n$-dimensional diagonal
tensor, with its diagonal entry $\mathcal{D}_{ii\cdots i}$ being $d_{i}$, the
degree of vertex $i$, for all $i\in\lbrack n]$. Then $\mathcal{Q(\mathcal{H}%
)}=\mathcal{D(\mathcal{H})}+\mathcal{A(\mathcal{H})}$ is the signless
Laplacian tensor of the hypergraph $\mathcal{H}$.
\end{definition}

The following general product of tensors, is defined in \cite{Shao} by Shao,
which is a generalization of the matrix case.

\begin{definition}
Let $\mathcal{A}\in\mathbb{C}^{n_{1}\times n_{2}\times\cdots\times n_{2}}$ and
$\mathcal{B}\in\mathbb{C}^{n_{2}\times n_{3}\times\cdots\times n_{k+1}}$ be
order $m\geq2$ and $k\geq1$ tensors, respectively. The product $\mathcal{AB}$
is the following tensor $\mathcal{C}$ of order $(m-1)(k-1)+1$ with entries:
\begin{equation}
\mathcal{C}_{i\alpha_{1}\cdots\alpha_{m-1}}=\sum_{i_{2},\cdots,i_{m}\in\lbrack
n_{2}]}\mathcal{A}_{ii_{2}\cdots i_{m}}\mathcal{B}_{i_{2}\alpha_{1}}%
\cdots\mathcal{B}_{i_{m}\alpha_{m-1}}. \label{1}%
\end{equation}
Where $i\in\lbrack n],\alpha_{1},\cdots,\alpha_{m-1}\in\lbrack n_{3}%
]\times\cdots\times\lbrack n_{k+1}]$.
\end{definition}

Let $\mathcal{T}$ be an order $k$ dimension $n$ tensor, let $x=(x_{1}%
,\cdot\cdot\cdot,x_{n})^{T}\in\mathbb{C}^{n}$ be a column vector of dimension
$n$. Then by (1) $\mathcal{T}x$ is a vector in $\mathbb{C}^{n}$ whose $i$th
component is as the following%

\begin{equation}
(\mathcal{T}x)_{i}=\sum_{i_{2},\cdots,i_{k}=1}^{n}\mathcal{T}_{ii_{2}\cdots
i_{k}}x_{i_{2}}\cdots x_{i_{k}}. \label{2}%
\end{equation}

Let $x^{[k]}=(x_{1}^{k}, \cdots, x_{n}^{k})^{T}$. Then (see \cite{ChangPZ}
\cite{Qi2014}) a number $\lambda\in\mathbb{C}$ is called an eigenvalue of the
tensor $\mathcal{T}$ if there exists a nonzero vector $x \in\mathbb{C}^{n}$
satisfying the following eigenequations%

\begin{equation}
\mathcal{T}x=\lambda x^{[k-1]}, \label{3}%
\end{equation}
and in this case, $x$ is called an eigenvector of $\mathcal{T}$ corresponding
to eigenvalue $\lambda$.

An eigenvalue of $\mathcal{T}$ is called an H-eigenvalue, if there exists a
real eigenvector corresponding to it (\cite{Qi2014}). The maximal absolute
value of eigenvalues of $\mathcal{T}$ is called the spectral radius of
$\mathcal{T}$ denoted by $\rho(\mathcal{T})$ (see \cite{Shao2}\ ).

In \cite{FGH}, the weak irreducibility of nonnegative tensors was defined. It
was proved in \cite{FGH} and \cite{Yang2014} that a $k$-uniform hypergraph
$\mathcal{H}$ is connected if and only if its adjacency tensor $\mathcal{A(H)}%
$ (and so $\mathcal{Q(H)}$) is weakly irreducible. They furthered proved the
following results, which implies that $\rho(\mathcal{T})$ is an H-eigenvalue
of $\mathcal{T}$ under some conditions.

\begin{lemma}
\cite{FGH}\cite{Yang2014} \label{Perron} Let $\mathcal{T}$ be a nonnegative
tensor. Then $\rho(\mathcal{T})$ is an H-eigenvalue of $\mathcal{T}$ with a
nonnegative eigenvector. Furthermore, if $\mathcal{T}$ is weakly irreducible,
then $\rho(\mathcal{T})$ has a positive eigenvector.
\end{lemma}

Let $\mathcal{T}$ be a tensor of order $k$ and dimension $n$. For
$i=1,2,\cdots,n$, denote by%
\begin{equation}
r_{i}(\mathcal{T})=\sum_{i_{2},\cdots,i_{k}\in\lbrack n]}\mathcal{T}%
_{ii_{2}\cdots i_{k}}. \label{4}%
\end{equation}

For a nonnegative tensor $\mathcal{T}$ the following bound for $\rho
(\mathcal{T})$ in terms of $r_{i}(\mathcal{T})$ was proposed in
\cite{Yang2010}, and the conditions for the equal cases were studied in
\cite{KhanFan}.

\begin{lemma}
\cite{KhanFan}\cite{Yang2010} \label{rowsum} Let $\mathcal{T}$ be a
nonnegative tensor of dimension $n$. We have
\begin{equation}
\min_{1\leq i\leq n}r_{i}(\mathcal{T})\leq\rho(\mathcal{T})\leq\max_{1\leq
i\leq n}r_{i}(\mathcal{T}).
\end{equation}
Moreover, if $\mathcal{T}$ is weakly irreducible, then the equality in (5)
holds if and only if $r_{1}(\mathcal{T})=\cdots=r_{n}(\mathcal{T})$.
\end{lemma}

For adjacency tensor $\mathcal{A(H)}$ of $k$-uniform hypergraph $\mathcal{H}$
we have $r_{i}(\mathcal{A(H)})=d_{i},$ where $d_{i}$ is the degree of vertex
$i.$ Hence Lemma \ref{rowsum} implies the following result, which is an analog
of a classical theorem in spectral graph theory.

\begin{corollary}
\cite{Cooper} \label{degree} Let $\mathcal{H}$ be a $k$-uniform hypergraph
with maximum degree $\Delta.$ Then $\rho(\mathcal{A}(\mathcal{H}))\leq\Delta.$
\end{corollary}

In this note we first give a bound on $\rho(\mathcal{A}(\mathcal{H}))$ in
terms of degrees of vertices, which improves the bound as shown in Corollary
\ref{degree}. Some bounds on $\rho(\mathcal{Q}(\mathcal{H}))$ are also proved.

\section{Several bounds on $\rho(\mathcal{A}(\mathcal{H}))$ and $\rho
(\mathcal{Q}(\mathcal{H}))$}

Some techniques of this note are based on the facts that diagonal similar
tensors have the same spectra (see \cite{Shao}).

\begin{definition}
\cite{Shao} \cite{Yang2014} Let $\mathcal{A}$ and $\mathcal{B}$ be two order
$k$ dimension $n$ tensors. Suppose that there exists a nonsingular diagonal
matrix $D$ of order $n$ such that $\mathcal{B}=D^{k-1}\mathcal{A}D$, then
$\mathcal{A}$ and $\mathcal{B}$ are called diagonal similar.
\end{definition}

Theorem 2.1 of \cite{Shao} implied the following result for similar tensors,
thus for diagonal similar tensors.

\begin{lemma}
\label{similar} \cite{Shao} Let $\mathcal{A}$ and $\mathcal{B}$ be two order
$k$ dimension $n$ similar tensors. Then $\mathcal{A}$ and $\mathcal{B}$ have
the same spectra.
\end{lemma}

Now we introduce a special class of hypergraphs, whose spectral radius of the
adjacency tensor can be determined by Theorem \ref{main 1}. Let $\mathcal{G}%
_{0}$ be a $d$-regular $(k-1)$-uniform hypergraph on $n-1$ vertices. If
$\mathcal{G}$ is obtained from $\mathcal{G}_{0}$ by adding a new vertex $v$ to
each edge of $\mathcal{G}_{0}$, then we may call that $\mathcal{G}$ is a
blow-up of $\mathcal{G}_{0}$ and write $\mathcal{G}=\mathcal{G}_{0}(v)$.
Obviously, $\mathcal{G}$ is a $k$-uniform hypergraph on $n$ vertices with
$d_{v}=|E(\mathcal{G})|=|E(\mathcal{G}_{0})|$ and $d_{u}=d$ for any
$u\in(V(\mathcal{G})\setminus\{v\}).$ Let $K_{k-1}^{k-1}$ be the
($k-1)$-uniform hypergraph on $k-1$ vertices, and $tK_{k-1}^{k-1}$ be $t$
disjoint unions of $K_{k-1}^{k-1}.$ For example, the hyperstar $\mathcal{S}%
_{t(k-1)+1,\text{ }k}$ (see \cite{Hu2}) is a blow-up of $tK_{k-1}^{k-1}.$

\begin{theorem}
\label{main 1}Let $\mathcal{H}$ be a $k$-uniform hypergraph on $n$ vertices
with degree sequence $d_{1}\geq d_{2}\geq\cdots\geq d_{n}$. Let
$\mathcal{A(\mathcal{H})}$ be the adjacency tensor of $\mathcal{H}$. Then
\[
\rho(\mathcal{A}(\mathcal{H}))\leq d_{1}^{\frac{1}{k}}d_{2}^{1-\frac{1}{k}}.
\]
Equality holds if and only if $\mathcal{H}$ is a regular hypergraph, or
$\mathcal{H}\ $is a blow-up$\ $of some regular hypergraph.
\end{theorem}

\begin{proof}
Write $\mathcal{A=A}(\mathcal{H})$ for short.

(1). If $d_{1}=d_{2}$, by Lemma \ref{rowsum}, we have
\[
\rho(\mathcal{A})\leq\max_{1\leq i\leq n}r_{i}(\mathcal{A})=\max_{1\leq i\leq
n}d_{i}=d_{1}=d_{1}^{\frac{1}{k}}d_{2}^{1-\frac{1}{k}}.
\]
Equality holds if and only if $r_{i}(\mathcal{A})$ is a constant. So
$\mathcal{H}$ is a regular hypergraph.

(2). Now we suppose that $d_{1}>d_{2}$ holds. If $P$ is a diagonal matrix,
then by (1), we have
\[
(P^{-(k-1)}\mathcal{A}P)_{i_{1}i_{2}\cdots i_{k}}=P_{i_{1}i_{1}}%
^{-(k-1)}\mathcal{A}_{i_{1}i_{2}\cdots i_{k}}P_{i_{2}i_{2}}\cdots
P_{i_{k}i_{k}}.
\]
Now take $P=diag(x,1,\cdots,1)$ with $x>1$. Then we have%
\[%
\begin{split}
r_{1}(P^{-(k-1)}\mathcal{A}P)  &  =\sum_{i_{2},\cdots,i_{k}\in\lbrack
n]}(P^{-(k-1)}\mathcal{A}P)_{1i_{2}\cdots i_{k}}\\
&  =\sum_{i_{2},\cdots,i_{k}\in\lbrack n]}P_{11}^{-(k-1)}\mathcal{A}%
_{1i_{2}\cdots i_{k}}P_{i_{2}i_{2}}\cdots P_{i_{k}i_{k}}\\
&  =\frac{1}{x^{k-1}}\sum_{i_{2},\cdots,i_{k}\in\lbrack n]}\mathcal{A}%
_{1i_{2}\cdots i_{k}}\\
&  =\frac{d_{1}}{x^{k-1}}.
\end{split}
\]
Denote by $d_{\{1,i\}}$ the number of edges, which contain vertices both $1$
and $i,$ i.e.,
\[
d_{\{1,i\}}=|\{e_{j}:\{1,i\}\subset e_{j}\in E(\mathcal{H})\}|.
\]
For $2\leq i\leq n$, we have%

\[%
\begin{split}
r_{i}(P^{-(k-1)}\mathcal{A}P)  &  =\sum_{i_{2},\cdots,i_{k}\in\lbrack
n]}(P^{-(k-1)}\mathcal{A}P)_{ii_{2}\cdots i_{k}}\\
&  =\sum_{i_{2},\cdots,i_{k}\in\lbrack n]}P_{ii}^{-(k-1)}\mathcal{A}%
_{ii_{2}\cdots i_{k}}P_{i_{2}i_{2}}\cdots P_{i_{k}i_{k}}\\
&  =\sum
_{\begin{subarray}{i} i_2,\cdots ,i_k\in [n]\\1\in \{i_2,\cdots ,i_k\} \end{subarray}}%
P_{ii}^{-(k-1)}\mathcal{A}_{ii_{2}\cdots i_{k}}P_{i_{2}i_{2}}\cdots
P_{i_{k}i_{k}}+\sum
_{\begin{subarray}{i} i_2,\cdots ,i_k\in [n]\\1 \not \in \{i_2,\cdots ,i_k\} \end{subarray}}%
P_{ii}^{-(k-1)}\mathcal{A}_{ii_{2}\cdots i_{k}}P_{i_{2}i_{2}}\cdots
P_{i_{k}i_{k}}\\
&  =xd_{\{1,i\}}+d_{i}-d_{\{1,i\}}\\
&  \leq xd_{i}\\
&  \leq xd_{2}.
\end{split}
\]
Noting that $d_{1}>d_{2},$ if we take
\[
x=(\frac{d_{1}}{d_{2}})^{\frac{1}{k}},
\]
then $x>1,$ and
\[
r_{1}(P^{-(k-1)}\mathcal{A}P)=d_{1}^{\frac{1}{k}}d_{2}^{1-\frac{1}{k}},
\]
for $2\leq i\leq n,$%
\[
r_{i}(P^{-(k-1)}\mathcal{A}P)\leq xd_{2}=d_{1}^{\frac{1}{k}}d_{2}^{1-\frac
{1}{k}}.
\]
Thus for each $1\leq i\leq n,$ we have
\[
r_{i}(P^{-(k-1)}\mathcal{A}P)\leq d_{1}^{\frac{1}{k}}d_{2}^{1-\frac{1}{k}}.
\]
\newline Then by Lemma \ref{rowsum},
\[
\rho(P^{-(k-1)}\mathcal{A}P)\leq\max_{1\leq i\leq n}r_{i}(P^{-(k-1)}%
\mathcal{A}P)=d_{1}^{\frac{1}{k}}d_{2}^{1-\frac{1}{k}}.
\]
\newline Furthermore, by Lemma \ref{similar}, we have
\begin{equation}
\rho(\mathcal{A})=\rho(P^{-(k-1)}\mathcal{A}P)\leq d_{1}^{\frac{1}{k}}%
d_{2}^{1-\frac{1}{k}}. \label{bound}%
\end{equation}

If the equality in (\ref{bound}) holds we have $d_{\{1,i\}}=d_{i}$, and
$d_{2}=d_{3}=\cdots=d_{n}$. The condition $d_{\{1,i\}}=d_{i}$ implies that any
edge containing vertex $i$ contains vertex $1,$ so $d_{1}$ equals to the
number of edges of $\mathcal{H}$. Concerning that $d_{2}=d_{3}=\cdots=d_{n},$
then $\mathcal{H\ }$is a blow-up of a $d_{2}$-regular and $(k-1)$-uniform hypergraph.

On the other hand, if $\mathcal{H=H}_{0}(v),$ where $\mathcal{H}_{0}$ is a
$d_{2}$-regular and $(k-1)$-uniform hypergraph, we take
\[
P=diag((\frac{d_{1}}{d_{2}})^{\frac{1}{k}},1,\cdots,1).
\]
Then for each $1\leq i$ $\leq n,$ we have
\[
r_{i}(P^{-(k-1)}\mathcal{A}P)=d_{1}^{\frac{1}{k}}d_{2}^{1-\frac{1}{k}},
\]
and Lemma \ref{rowsum} and Lemma \ref{similar} implies that
\[
\rho(\mathcal{A})=\rho(P^{-(k-1)}\mathcal{A}P)=d_{1}^{\frac{1}{k}}%
d_{2}^{1-\frac{1}{k}}.
\]

\end{proof}

\begin{lemma}
\label{increasing}Let $\mathcal{A}$ and $\mathcal{B}$ be two order $k$
dimension $n$ tensors satisfying $|\mathcal{A}|$ $\leq\mathcal{B},$ where
$\mathcal{B}$ is weakly irreducible. Let $\lambda$ be an eigenvalue of
$\mathcal{A}$. \ Then $|\lambda|\leq\rho(\mathcal{B}).$
\end{lemma}

\begin{corollary}
\label{lower bound}Let $\mathcal{H}$ be a connected $k$-uniform hypergraph on
$n$ vertices with degree sequence $d_{1}\geq\cdots\geq d_{n}$. Let
$\mathcal{Q(\mathcal{H})}$ be the signless Laplacian tensor of $\mathcal{H}$. Then

(1) $\rho(\mathcal{Q}(\mathcal{H}))\geq d_{1};$

(2) $\rho(\mathcal{Q}(\mathcal{H}))\leq d_{1}+d_{1}^{\frac{1}{k}}%
d_{2}^{1-\frac{1}{k}},$ equality holds if and only if $\mathcal{H}$ is a
regular hypergraph.
\end{corollary}

\begin{proof}
(1) Noting that $\mathcal{D}(\mathcal{H})\leq\mathcal{Q}(\mathcal{H}),$ Lemma
\ref{increasing} implies that $d_{1}=\rho(\mathcal{D}(\mathcal{H}))\leq
\rho(\mathcal{Q}(\mathcal{H})).$

(2) Let $\mathcal{I\ }$be the unit tensor and $\mathcal{D}^{\prime}%
=d_{1}\mathcal{I}$. Then $|\mathcal{Q}(\mathcal{H})|\leq\mathcal{D}^{\prime
}+\mathcal{A(H)}$, and so by Lemma \ref{increasing}, we have $\rho
(\mathcal{Q}(\mathcal{H}))\leq\rho({\mathcal{D}^{\prime}+\mathcal{A(H)}})$. It
is not difficult to see that $\rho(\mathcal{Q}(\mathcal{H}))=\rho
({\mathcal{D}^{\prime}+\mathcal{A(H)}})$ if and only if $d_{1}=d_{n}$. Thus
\[
\rho(\mathcal{Q}(\mathcal{H}))\leq\rho(\mathcal{D}^{\prime}+\mathcal{A(H)}%
)=\rho(\mathcal{D}^{\prime})+\rho(\mathcal{A(H)})=d_{1}+\rho(\mathcal{A(H)}%
)\leq d_{1}+d_{1}^{\frac{1}{k}}d_{2}^{1-\frac{1}{k}}.
\]
Equality holds if and only if $d_{1}=d_{n}$, namely, $\mathcal{H}$ is a
regular hypergraph.
\end{proof}

\begin{theorem}
\label{signless Laplacian} Let $\mathcal{H}$ be a connected $k$-uniform
hypergraph on $n$ vertices, and $b_{i}>0$ for each $1\leq i\leq n$. Then,%

\[
\rho(\mathcal{Q}(\mathcal{H}))\leq\max_{e\in E(\mathcal{H})}\max_{\{i,\text{
}j\}\subseteq e}\frac{d_{i}+d_{j}+\sqrt{(d_{i}-d_{j})^{2}+4b_{i}^{\prime}%
b_{j}^{\prime}}}{2},
\]
where $b_{p}^{\prime}=b_{p}^{-(k-1)}\sum\limits_{\{p,p_{2},\cdots,p_{k}\}\in
E(\mathcal{H})}b_{p_{2}}\cdots b_{p_{k}}$ for any $1\leq p\leq n.$
\end{theorem}

\begin{proof}
Write $\mathcal{Q}(\mathcal{H})=\mathcal{Q}$ and $\rho(\mathcal{Q}%
(\mathcal{H}))=\rho$ for short. Let $B=diag(b_{1},b_{2},\cdots,b_{n})$ and
$b_{i}>0$ for any $1\leq i\leq n.$ By Lemma \ref{similar} we know that
$\rho(B^{-(k-1)}\mathcal{Q}B)=\rho$. By (1) we have have%
\[
(B^{-(k-1)}\mathcal{Q}B)_{i_{1}i_{2}\cdots i_{k}}=B_{i_{1}i_{1}}%
^{-(k-1)}\mathcal{Q}_{_{i_{1}i_{2}\cdots i_{k}}}B_{i_{2}i_{2}}\cdots
B_{i_{k}i_{k}}.
\]
Since $\mathcal{H}$ is connected, the tensor $\mathcal{Q}$ and so the tensor
$B^{-(k-1)}\mathcal{Q}B$ is weakly irreducible. By Lemma \ref{Perron} we know
that $\rho$ is an H-eigenvalue of $B^{-(k-1)}\mathcal{Q}B$ and there exists a
positive eigenvector corresponding to $\rho$, denoted by $x$. We may suppose
that $x_{i}=1,x_{p}\leq1$ for any vertex $p$ different from $i$. Let
\[
x_{j}=\max\{x_{p}:\{i,p\}\subseteq e\in E(\mathcal{H})\}.
\]
From the definitions of eigenvalue and eigenvector (see (3)), we have%

\[
(B^{-(k-1)}\mathcal{Q}B)x=\rho x^{[k-1]}.
\]
For any vertex $p$ we have%
\[
((B^{-(k-1)}\mathcal{Q}B)x)_{p}=\rho x_{p}^{k-1}.
\]
By (1) we have
\[
\sum\limits_{p_{2},\cdots,p_{k}\in\lbrack n]}(B^{-(k-1)}\mathcal{Q}%
B)_{pp_{2}\cdots p_{k}}x_{p_{2}}\cdots x_{p_{k}}=\rho x_{p}^{k-1},
\]
then,%
\[
d_{p}x_{p}^{k-1}+\sum\limits_{\{p,p_{2},\cdots,p_{k}\}\in E(\mathcal{H})}%
b_{p}^{-(k-1)}b_{p_{2}}\cdots b_{p_{k}}x_{p_{2}}\cdots x_{p_{k}}=\rho
x_{p}^{k-1}.
\]
Hence we have%
\begin{equation}
(\rho-d_{p})x_{p}^{k-1}=b_{p}^{-(k-1)}\sum\limits_{\{p,p_{2},\cdots,p_{k}\}\in
E(\mathcal{H})}b_{p_{2}}\cdots b_{p_{k}}x_{p_{2}}\cdots x_{p_{k}}. \label{p}%
\end{equation}
Recall that for any vertex $p$
\[
b_{p}^{\prime}=b_{p}^{-(k-1)}\sum\limits_{\{p,p_{2},\cdots,p_{k}\}\in
E(\mathcal{H})}b_{p_{2}}\cdots b_{p_{k}}.
\]
Now take $p=i$ in (\ref{p}), then we obtain%
\[%
\begin{split}
\rho-d_{i}  &  =b_{i}^{-(k-1)}\sum\limits_{\{i,i_{2},\cdots,i_{k}\}\in
E(\mathcal{H})}b_{i_{2}}\cdots b_{i_{k}}x_{i_{2}}\cdots x_{i_{k}}\\
&  \leq b_{i}^{-(k-1)}\sum\limits_{\{i,i_{2},\cdots,i_{k}\}\in E(\mathcal{H}%
)}b_{i_{2}}\cdots b_{i_{k}}x_{j}^{k-1},\\
&  =b_{i}^{\prime}x_{j}^{k-1}.
\end{split}
\]
And take $p=j$ in (\ref{p}), then we have
\[%
\begin{split}
(\rho-d_{j})x_{j}^{k-1}  &  =b_{j}^{-(k-1)}\sum\limits_{\{j,j_{2},\cdots
,j_{k}\}\in E(\mathcal{H})}b_{j_{2}}\cdots b_{j_{k}}x_{j_{2}}\cdots x_{j_{k}%
}\\
&  \leq b_{j}^{-(k-1)}\sum\limits_{\{j,j_{2},\cdots,j_{k}\}\in E(\mathcal{H}%
)}b_{j_{2}}\cdots b_{j_{k}},\\
&  =b_{j}^{\prime}.
\end{split}
\]
Now we obtain%
\[
\rho-d_{i}\leq b_{i}^{\prime}x_{j}^{k-1}\text{ and }(\rho-d_{j})x_{j}%
^{k-1}\leq b_{j}^{\prime}.
\]
Noting that $\rho\geq d_{p}$ (see Corollary \ref{lower bound}), multiplying
the left and right sides of the two inequalities, respectively, we have,%
\[
(\rho-d_{i})(\rho-d_{j})x_{j}^{k-1}\leq b_{i}^{\prime}b_{j}^{\prime}%
x_{j}^{k-1}.
\]
Thus we have,%
\[
\rho^{2}-(d_{i}+d_{j})\rho+d_{i}d_{j}-b_{i}^{\prime}b_{j}^{\prime}\leq0,
\]
and then%
\[
\rho\leq{\frac{d_{i}+d_{j}+\sqrt{(d_{i}-d_{j})^{2}+4b_{i}^{\prime}%
b_{j}^{\prime}}}{2}.}%
\]
So we have proved that
\[
\rho(\mathcal{Q}(\mathcal{H}))\leq\max_{e\in E(\mathcal{H})}\max_{\{i,\text{
}j\}\subseteq e}\frac{d_{i}+d_{j}+\sqrt{(d_{i}-d_{j})^{2}+4b_{i}^{\prime}%
b_{j}^{\prime}}}{2}.
\]

The proof is completed.
\end{proof}

For each $i$, if we take $b_{i}=1$, then $b_{i}^{\prime}=d_{i}$ in Theorem
\ref{signless Laplacian}, we may obtain the following result.

\begin{corollary}
Let $\mathcal{H}$ be a $k$-uniform hypergraph. Then we have
\[
\rho(Q(\mathcal{H}))\leq\max_{e\in E(\mathcal{H})}\max_{\{i,\text{
}j\}\subseteq e}(d_{i}+d_{j}).
\]

\end{corollary}

For a vertex $i$ of the $k$-uniform hypergraph $\mathcal{H}$, denote by
\[
m_{i}=\frac{\sum\limits_{\{i,i_{2},\cdots,i_{k}\}\in E(\mathcal{H})}d_{i_{2}%
}\cdots d_{i_{k}}}{d_{i}^{k-1}},
\]
which is a generalization of the average of degrees of vertices adjacent to
$i$ of\ the simple graph.

For each $i$, if we take $b_{i}=d_{i}$, then $b_{i}^{\prime}=m_{i}$ in Theorem
\ref{signless Laplacian}, we may obtain the following result.

\begin{corollary}
Let $\mathcal{H}$ be a $k$-uniform hypergraph. Then we have
\[
\rho(Q(\mathcal{H}))\leq\max_{e\in E(\mathcal{H})}\max_{\{i,\text{
}j\}\subseteq e}\frac{d_{i}+d_{j}+\sqrt{(d_{i}-d_{j})^{2}+4m_{i}m_{j}}}{2}.
\]

\end{corollary}

\begin{remark}
Take $B=diag(1,1,\cdots,1)$ and use the similar arguments as that in Theorem
\ref{signless Laplacian} for the tensor $B^{-(k-1)}\mathcal{A(H)}B$, then we
may obtain that
\[
\rho(\mathcal{A(H)})\leq\max_{e\in E(\mathcal{H})}\max_{\{i,\text{
}j\}\subseteq e}\sqrt{d_{i}d_{j}}.
\]

\end{remark}

If take $B=diag(d_{1},d_{2},\cdots,d_{n})$, then we may prove that
\[
\rho(\mathcal{A(H)})\leq\max_{e\in E(\mathcal{H})}\max_{\{i,\text{
}j\}\subseteq e}\sqrt{m_{i}m_{j}}.
\]

%参考文献中注意以下几点：%
%(1) 文献结束后不要句点;人名缩写需空开; 卷号与年份之间空开，如：2 (2003)； 参考文献末一般按卷期年页的顺序排，如：7(13), 2011, 4963-4971%
%(2) 文章名句首字母大写，其他小写；期刊名[J]或会议名各单词首字母大写，介词冠词除外%
%(3) 文献中作者要名前，姓后，姓除首字母大写外，其他小写。如有缩略名，缩略名后要加句点。例如 M. Treumuth, H. Fan.%

\end{document}